\documentclass[11pt]{article}

\usepackage{amsthm,amssymb,latexsym,graphics,amscd,amsmath,amscd,enumerate,color}

\newcommand{\C}{\mathbb{C}} \newcommand{\PP}{\mathbb{P}}
\newcommand{\Ccl}{\mathbb{C}\ell}

\newcommand{\mb}{\mathbb}
\newcommand{\mc}{\mathcal}
\newcommand{\R}{\mathbb{R}} \newcommand{\Z}{\mathbb{Z}}

\newcommand{\rk}{\operatorname{rk}}
\newcommand{\Sc}{\mb{S}_{\mb{C}}}
\newcommand{\Sp}{\mbox{Spin}}
\newcommand{\Spc}{\mbox{Spin}^{\mb{C}}}
\newcommand{\tr}{\mbox{tr}}
\newcommand{\vol}{\operatorname{vol}}

\newtheorem{teo}{Theorem}
\newtheorem{lem}[teo]{Lemma}
\newtheorem{prop}[teo]{Proposition}

\newtheorem*{cor}{Corollary}

\newtheorem*{mthm}{Main Theorem}

\begin{document}

\title{On the spectrum of the twisted Dolbeault Laplacian over K\"ahler manifolds}
\author{Marcos Jardim \\ IMECC - UNICAMP \\
Departamento de Matem\'atica \\ Caixa Postal 6065 \\
13083-970 Campinas-SP, Brazil \\ and \\ Rafael F. Le\~ao \\ Departamento de Matem\'atica \\
Universidade Federal do Paran\'a \\ Caixa Postal 019081 \\ 81531-990 Curitiba-PR, Brazil}

\maketitle

\begin{abstract}
We use Dirac operator techniques to a establish sharp lower bound for the first eigenvalue of the Dolbeault Laplacian twisted by Hermitian-Einstein connections on a vector bundle of negative degree over compact K\"ahler manifolds.

\end{abstract}


\section{Introduction}

Many results on the spectrum of the Hodge Laplacian on differential forms (see \cite{Bo} and the references therein),
and on the spectrum of the Dirac operator (see for instance \cite{A} and the book \cite{F}) can be found in the literature. In particular, the relation between the eigenvalues of the Dirac operator and those of the Dolbeault Laplacian on a K\"ahler manifold is considered in \cite{Ki1,Ki2}.

However, there are very few results available for the case of twisted operators, i.e. when some additional connection $\nabla_A$ on a vector bundle $E$ is considered. Recently, some results concerning twisted Laplacian and Dirac operators on Riemann surfaces of constant curvature appeared in \cite{AP,Pr1}. In this article, we give a sharp estimate for the first eigenvalue of the twisted Dolbeault Laplacian on K\"ahler manifolds and for the first nonzero eigenvalue of the twisted complex Dirac operator on Riemann surfaces. 

Let $M$ be a compact K\"ahler manifold of complex dimension $n$ and let $E$ be a holomorphic Hermitian vector bundle. We can endow $E$ with a connection $\nabla_A$ compatible both with the Hermitian and holomorphic structures; this is the so-called Chern connection on $E$. One then has the decomposition:
\begin{equation}
  \nabla_A = \partial_A + \bar{\partial}_A ~~.
\end{equation}

If we denote the space of sections $\Gamma(E \otimes \wedge^{p,q}M)$ by $\Omega^{p,q}(E)$, then $\partial_A$ and $\bar{\partial}_A$ are first order differential operators acting as follows ($p,q=0,\dots,n$)
\begin{equation}
  \begin{split}
    \partial_A : \Omega^{p,q}(E) \rightarrow \Omega^{p+1,q}(E) \\
    \bar{\partial}_A : \Omega^{p,q}(E) \rightarrow \Omega^{p,q+1}(E)
  \end{split}
\end{equation}
Using the Hermitian structure of $E$ and the metric of $M$ we can define their formal adjoints
\begin{equation}
  \begin{split}
    \partial_A^* : \Omega^{p,q}(E) \rightarrow \Omega^{p-1,q}(E) \\
    \bar{\partial}_A^* : \Omega^{p,q}(E) \rightarrow \Omega^{p,q-1}(E)~.
  \end{split}
\end{equation}
These operators define a natural, second order differential operator on \linebreak $\Omega^{\bullet}(E) = \oplus_{p,q} \Omega^{p,q}(E)$, the so called Dolbeault Laplacian:
\begin{equation}
  \Delta_{\bar{\partial}} = \bar{\partial}_A \bar{\partial}_A^* + \bar{\partial}_A^* \bar{\partial}_A ~.
\end{equation}
Restricted to $\Omega^{0}(E)$, the Dolbeault Laplacian simplifies to $\bar{\partial}_A^* \bar{\partial}_A$. This is the operator we shall concentrate on. Its kernel, which consists of the holomorphic sections of $E$, is known to vanish under some circumstances, and then it makes sense to ask for lower bounds for its spectrum.

One class of holomorphic vector bundles for which $\ker \bar{\partial}_A^* \bar{\partial}_A=H^0(E)=0$ are stable holomorphic bundles of negative degree. Recall that the degree of a complex vector bundle over a K\"ahler manifold is defined as follows:
$$ \deg(E) = \int_M c_1(E)\wedge \omega^{n-1} ~~, $$
where $\omega$ is the K\"ahler form and $n$ is the complex dimension of $M$. Also, a holomorphic bundle is stable if and only if it admits a Hermitian-Einstein connection, i.e. a compatible connection $\nabla_A$ whose curvature $F_A$ satisfies $i\Lambda F_A = c \mb{I}_E$ where $c$ is a topological constant equal to
$$\frac{2 \pi \deg{E}}{(n-1)! \rk(E) \vol(M)} ~,$$
and $\mb{I}_E$ is the identity endomorphism of the bundle $E$. Here, $\Lambda$ denotes contraction by the K\"ahler form, as usual.

Using the K\"ahler identities, one can easily establish the following \linebreak Weitzenb\"ock type formula \cite[Lemma 6.1.7]{DK}:
\begin{equation} \label{eq_kah_ide}
  \bar{\partial}_A^* \bar{\partial}_A = \frac{1}{2} \nabla_A^* \nabla_A - \frac{i}{2} \Lambda F_A~~,
\end{equation}
Thus if $s \in \Gamma(E)$ is a section such that $\bar{\partial}_A^* \bar{\partial}_A s = \lambda s$ and $\nabla_A$ is a Hermitian-Einstein connection, we can use equation (\ref{eq_kah_ide}) to obtain the inequality:
\begin{equation}
\lambda \geq \frac{- \pi \deg(E)}{(n-1)! \rk(E) \vol(M)}
\end{equation}
This estimate makes sense for stable bundles of negative degree over arbitrary K\"ahler manifolds. However, such estimate is never sharp. Indeed, notice that equality holds precisely when $\nabla_A \psi = 0$ hence $\bar{\partial}_A \psi = 0$, so $\bar{\partial}_A^* \bar{\partial}_A \psi = 0$, leading to a contradiction, since $E$
had no holomorphic sections.

This fact suggests that a better estimate for the eigenvalues of the twisted Dolbeault Laplacian $\bar{\partial}_A^* \bar{\partial}_A$ on sections of a bundle of negative degree admitting an Hermitian-Einstein connection can be obtained. 

We consider the canonical $\Spc$ structure associated to complex manifolds and the associated twisted complex Dirac operator. This operator can be viewed as a square root for the Dolbeault Laplacian, and we show in Section \ref{w} how the Weitzenb\"ock formula relating these two operators generalizes the formula (\ref{eq_kah_ide}) obtained from the K\"ahler identities. This generalized expression is subsequently used in Section \ref{p} to prove our main result:

\begin{mthm} \label{teo_dol}
Let $M$ be a compact K\"ahler manifold of complex dimension $n$, and let $E$ be a holomorphic Hermitian vector bundle of negative degree with a compatible connection $\nabla_A$. Suppose that $\nabla_A$ satisfies the Hermitian-Einstein condition. If $\lambda$ is an eigenvalue for the operator $\bar{\partial}_A^* \bar{\partial}_A: \Omega^0(E) \to \Omega^0(E)$, then
\begin{equation}\label{lb}
  \lambda \geq -\frac{2n}{2n-1}\frac{\pi \deg(E)}{(n-1)! \rk(E) \vol(M)}~.
\end{equation}
\end{mthm}

In Section \ref{s} we will present examples in which the lower bound (\ref{lb}) is actually attained.

In \cite{A}, Atiyah pointed out that absolute upper bounds for the first eigenvalue of second order differential operators do not exist in general, by claiming that the first eigenvalue of the twisted trace Laplacian on a line bundle $L$ over a Riemann surface is proportional to the first Chern class of $L$. The lower bound of our Main Theorem makes Atiyah's claim more precise for the case of the Dolbeault Laplacian, and generalizes it to bundles of higher rank over K\"ahler manifolds of arbitrary dimension.

For Riemann surfaces (case $n=1$), our Theorem has an important consequence for the spectrum of the twisted complex and real Dirac operators.

\begin{cor}
Let $M$ be a Riemann surface, and let $E$ be a holomorphic Hermitian vector bundle of negative degree with a compatible connection $\nabla_A$ satisfying the Hermitian-Einstein condition. Then the nonzero eigenvalues $\mu$ of the twisted complex Dirac operator $D_A$ satisfy
\begin{equation}\label{lb-d}
  \mu \geq \sqrt{- \frac{4 \pi \deg(E)}{\rk(E) \vol(M)}}~.
\end{equation}
\end{cor}

\begin{cor}
Let $M$ be a Riemann surface of genus $g$, and let $E$ be a holomorphic Hermitian vector bundle of negative degree with a compatible connection $\nabla_A$ satisfying the Hermitian-Einstein condition. Then the nonzero eigenvalues $\nu$ of the twisted real Dirac operator ${\cal D}_A$ satisfy
\begin{equation}\label{lb-d-r}
\nu \geq \sqrt{\frac{4\pi(1-g)}{\vol(M)}- \frac{4 \pi \deg(E)}{\rk(E) \vol(M)}} \geq \sqrt{\frac{R_0}{2}- \frac{4 \pi \deg(E)}{\rk(E) \vol(M)}}~,
\end{equation}
where $R_0$ is the minimum of the scalar curvature of $M$.
\end{cor}

\paragraph{\bf Acknowledgments.}
The first named author is partially supported by the CNPq grant number 305464/2007-8
and the FAPESP grant number 2005/04558-0. The second author's research was supported by
a CAPES doctoral grant.


\section{Weitzenb\"ock formulas and  K\"ahler Identities} \label{w}

Let $M$ be a K\"ahler manifold with complex dimension $n$. As it is well know, c.f. \cite{Hi}, the spinor bundle associated to the canonical $\Spc$ structure of $M$ can be identified with the holomorphic forms of $M$, in other words, we have the identification $\Sc \simeq \wedge^{0,*}M$.

In this way, the spinors coupled to $(E,\nabla_A)$ can be identified with elements of $\Sc \otimes E$, so $\Omega^{0,*}(E)$ can be identified with the coupled spinors. Besides, we can consider the twisted complex Dirac operator $D_A$ on $\Sc \otimes E$, and if we identify $\Sc \otimes E$ with $\Omega^{0,*}(E)$ then this Dirac operator can be written as
\begin{equation}
  D_A = \sqrt{2} ( \bar{\partial}_A + \bar{\partial}_A^*) ~~.
\end{equation}
In particular this identity implies that $D_A^2 = 2\Delta_{\bar{\partial}}$.

The Dirac operator $D_A$ satisfies the the Weitzenb\"ock formula (c.f. \cite{La})
\begin{equation}
  D_A^2 = \nabla_{\tilde{A}}^* \nabla_{\tilde{A}} + \frac{1}{4} R + \frac{1}{2} \Omega_{\Sc} + F_A
\end{equation}
where $\nabla_{\tilde{A}}^* \nabla_{\tilde{A}}$ is the trace Laplacian associated to the tensor product connection
$\nabla_{\tilde{A}} = \nabla_S \otimes \mb{I}_E + \mb{I}_{\Sc} \otimes \nabla_A$, $R$ is the scalar curvature of $M$, $F_A$ is the curvature 2-form of $\nabla_A$ and $\Omega_{\Sc}$ is the curvature 2-form for some connection on the determinant bundle of the $\Spc$ structure. In principle, the connection on the determinant bundle of the $\Spc$ structure can be an arbitrary Hermitian connection, but since we are dealing with the canonical $\Spc$ structure associated to the complex structure of $M$, and the determinant bundle of this structure is just the anti-canonical bundle of $M$, $K_M^{-1} = (\wedge^{0,n}M)^*$. In this way, there is a natural connection to be used, namely the extension to $\wedge^{0,n}M$ of the Chern connection of $M$. As we will see, for this connection, $\Omega_{\Sc}$ has a nice description in terms of the Riemannian scalar curvature of $M$.

Note that the identification $\Sc \otimes E\simeq\Omega^{0,*}(E)$ can be used to describe explicitly the action of the Clifford Algebra $\Ccl(M)$ on $\Sc \otimes E$. If $\{ \xi^p , \bar{\xi}^p \}$ is an unitary frame for $T^*M \otimes \C$, then the action of $\Ccl(M)$ is given by
\begin{equation}
  \begin{split}
    \xi^k \cdot \left( s \otimes t \right) = \left( \sqrt{2} \xi^k \lrcorner s \right) \otimes t = \sqrt{2}
      c(\xi^k) s \otimes t \\
    \bar{\xi}^k \cdot \left( s \otimes t \right) = \left( - \sqrt{2} \bar{\xi}^k \wedge s \right) \otimes t = - 
      \sqrt{2} e(\bar{\xi}^k) s \otimes t
  \end{split}
\end{equation}
where $s \in \Gamma(\Sc)$ and $t \in \Gamma(E)$. With this in mind we have:

\begin{prop} \label{prop_aca_sec}
The Clifford action of an $(1,1)$-form $\alpha$ on sections of $E$, $\psi_0 \in \Omega^{0,0}(E)$, is explicit given by
\begin{equation}
  \alpha \cdot \psi_0 = -i ( \Lambda \alpha ) \psi_0
\end{equation}
where $\Lambda \alpha = \omega \lrcorner \alpha$ is the contraction of $\alpha$ by the K\"ahler form $\omega$.
\end{prop}

\begin{proof}
A 2-form acts on spinors through Clifford multiplication by means of the identification
\begin{equation}
  \alpha \wedge \beta = \frac{1}{2} \left( \alpha \beta - \beta \alpha \right)
\end{equation}
which identifies 2-forms with elements in the Clifford Algebra.

Writing the $(1,1)$-form $\alpha$ as $\alpha = \sum_{p,q} \alpha_{pq} \xi^p \wedge \bar{\xi}^q$, and noting that for a general element $\psi \in \Omega^{0,*}(E)$ we have
\begin{equation}
  \begin{split}
    \xi^k \wedge \bar{\xi}^l \cdot \psi &= \frac{1}{2}\left( \xi^k \bar{\xi}^l - \bar{\xi}^l \xi^k \right) \cdot \psi \\
    &= \left( -c( \xi^k ) e( \bar{\xi}^l ) + e( \bar{\xi}^l ) c( \xi^k ) \right) \psi
  \end{split}
\end{equation}
we immediately conclude that for a section $\psi_0 \in \Omega^{0,0}(E)$
\begin{equation}
  \begin{split}
    \xi^p \wedge \bar{\xi}^q \cdot \psi_0 &= \left( -c( \xi^p ) e( \bar{\xi}^q ) + e( \bar{\xi}^q ) c( \xi^p ) 
      \right) \psi_0 \\
    &= -c( \xi^p ) \left( e( \bar{\xi}^q ) \psi_0 \right) = \left( -c( \xi^p ) \bar{\xi}^q \right) \psi_0 
    = - \delta_{pq} \psi_0
  \end{split}
\end{equation}
from which it follows that
\begin{equation}
  \alpha \cdot \psi_0 = \left( \sum_k \alpha_{kk} \right) \psi_0 ~~.
\end{equation}

On the other hand, for the $(1,1)$-form $\alpha$, we have that
\begin{equation}
\begin{split}
\Lambda \alpha &=
\left( i \sum_k \xi^k \wedge \bar{\xi}^k \right) \lrcorner \left( \sum_{pq} \alpha_{pq} \xi^p \wedge \bar{\xi}^q \right)
= i \sum_k \xi^k \lrcorner \left( \bar{\xi}^k \lrcorner \sum_{pq} \alpha_{pq} \xi^p \wedge \bar{\xi}^q \right) \\
&= i \sum_k \xi^k \lrcorner \left( \sum_q \alpha_{kq} \bar{\xi}^q \right) = i \sum_k \alpha_{kk}
\end{split} \end{equation}
These two relations yield $\alpha \cdot \psi_0 = -i ( \Lambda \alpha ) \psi_0$, as desired.
\end{proof}

With this Proposition, we can write the Weitzenb\"ock formula restricted to sections of $E$ as follows:
\begin{equation}\label{cc}
  D_A^2|_{\Omega^0(E)} = 2 \bar{\partial}_A^* \bar{\partial}_A =
  \nabla_{\tilde{A}}^* \nabla_{\tilde{A}} + \frac{1}{4} R - 
  \frac{i}{2} \Lambda \Omega_{\Sc} - i \Lambda F_A
\end{equation}
The term $i\Lambda F_S$ has a nice interpretation.

\begin{prop} \label{prop_cur_esc}
Let $M$ be a K\"ahler manifold and consider on the anti-canonical line bundle, $K_M^{-1}$, the connection induced by the Chern connection of $M$. Let $\Omega_{\Sc}$ be the curvature 2-form of this connection, then we have
\begin{equation}
  i\Lambda \Omega_{\Sc} = \frac{R}{2}
\end{equation}
where $R$ is the Riemannian scalar curvature of $M$.
\end{prop}

\begin{proof}
We can look to a K\"ahler manifold $M$ as a Riemannian manifold with complex structure $(M,g,J)$ such that

\begin{equation}
  \begin{split}
    g(J u , J v) = g(u,v) \\
    \nabla J = 0
  \end{split}
\end{equation}
where $\nabla$ is the Levi-Civitta connection of $(M,g)$.

Using the Riemannian structure of $(M,g)$ we can define the curvature operator $\mc{R}$ in the usual way. But the fact that the complex structure $J$ is compatible with $g$ implies that

\begin{equation}
  \begin{split}
    \mc{R}(u,v,Jz,Jw) = \mc{R}(u,v,z,w) \\
    \mc{R}(Ju,Jv,z,w) = \mc{R}(u,v,z,w) \\
    \mc{R}(Ju,v,z,w) = - \mc{R}(u,v,z,w) \\
    \mc{R}(u,v,Jz,w) = - \mc{R}(u,v,z,w), \label{id_cur_kah}
  \end{split}
\end{equation}

Now we extend $\mc{R}$ to $T M \otimes \C$ by complex linearity, and note that this extension coincides with the curvature operator of the Chern connection of $M$, because the Chern connection is just the extension of the Levi-Civita connection by complex linearity. Writing $T M \otimes \C = T^{1,0}M \oplus T^{0,1}M$ we see that $T_{\R} M$ e $T^{1,0}M$ are isomorphic trough the application $\tilde{u} \mapsto u = \frac{1}{\sqrt{2}} ( \tilde{u} - i J \tilde{u} )$.

Using the $\C$-linearity of $\mc{R}$ we see that the only non-trivial terms of $\mc{R}$ on $TM \otimes \C$ are the terms of the form $\mc{R}(x,\bar{y},z,\bar{w})$ for vector fields $x,y,z,w \in T^{1,0}M$. With this in mind we define the Ricci tensor for a K\"ahler manifold in one of the following equivalent ways

\begin{equation}
  \begin{split}
    r(x,\bar{y}) &= \sum_{k=1}^n \mc{R}(x,\bar{y},\xi_k,\bar{\xi}_k) \\
    r(x,\bar{y}) &= \sum_{k=1}^n \mc{R}(\xi_k,\bar{\xi}_k,x,\bar{y}) \\
    r(x,\bar{y}) &= \sum_{k=1}^n \mc{R}(x,\bar{\xi}_k,\xi_k,\bar{y})
  \end{split}
\end{equation}

The Riemannian Ricci tensor can be recovered using the above definition. If we take two elements $u,v \in T^{1,0}M$ of the form $u = \frac{1}{\sqrt{2}}(\tilde{u} -i J \tilde{u})$ and \linebreak $v = \frac{1}{\sqrt{2}}(\tilde{v} -i J \tilde{v})$ then we see that

\begin{equation}
  r( u, \bar{v} ) = Ricc(\tilde{u},\tilde{v})+i Ricc(\tilde{u},J \tilde{v})
\end{equation}
where $Ricc$ is the Riemannian Ricci tensor of $(M,g)$.

With this relation we immediately see that the Riemannian scalar curvature of $(M,g)$ is given by

\begin{equation}
  R = 2 \sum_{k=1}^n r(\xi_k,\bar{\xi}_k) \label{eq_sc}
\end{equation}

Writing the curvature 2-form of the Chern connection of $M$, restricted to $T^{1,0}M$, as $F\mid_{T^{1,0}M} = \sum_{p,q} \Omega^p_q \xi_p \otimes \bar{\xi}^q$, and using the relation $\mc{R}(u,v) = v \lrcorner \left( u \lrcorner F \right)$ we can conclude that

\begin{equation}
  \Omega^h_k = - \sum_{pq} \mc{R}(\xi_p, \bar{\xi}_q, \xi_k, \bar{\xi}_h) \xi^p \wedge \bar{\xi}^q = 
    - \mc{R}_{p \bar{q} k \bar{h}} \xi^p \wedge \bar{\xi}^q
\end{equation}

This implies that

\begin{equation}
  \begin{split}
    \tr F &= \sum_k \Omega^k_k \\
    &= - \sum_{k p q} \mc{R}_{p \bar{q} k \bar{k}} \xi^p \wedge \bar{\xi}^q
  \end{split}
\end{equation}

Using the isomorphism $T^{1,0}M \simeq (T^{0,1}M)^*$ and the fact that the anti-canonical bundle coincides with $(\wedge^{0,n} T^{0,1}M)^*$ it is immediate that $\Omega_{\Sc} = \tr F$. This fact, the above identity and the equation (\ref{eq_sc}) we calculate

\begin{equation}
  \begin{split}
    \Lambda \left( \Omega_{\Sc} \right) = \Lambda \left( \tr F \right) &= \omega \lrcorner (\tr F) \\
    &= i \left( \sum_k \xi^k \wedge \bar{\xi}^k \right) \lrcorner \left( - \sum_{l p q} \mc{R}_{p \bar{q} l \bar{l}}
      \xi^p \wedge \bar{\xi}^q \right) \\
    &= - i \sum_{k l p q} \mc{R}_{p \bar{q} l \bar{l}} \delta^{kp} \delta^{kq} \\
    &= - i \sum_{k l} \mc{R}_{k \bar{k} l \bar{l}} = - i \sum_{k} \left( \sum_l \mc{R}_{k \bar{k} l \bar{l}} \right) \\
    &= - i \sum_k r(\xi_k , \bar{\xi}_k) =  - \frac{i}{2} R~,
  \end{split}
\end{equation}
as desired.
\end{proof}

Now (\ref{cc}) can be rewritten as
\begin{equation} \label{prop_ide_gen}
  \bar{\partial}_A^* \bar{\partial}_A = \frac{1}{2} \nabla_{\tilde{A}}^* \nabla_{\tilde{A}} - \frac{i}{2} \Lambda F_A
\end{equation}
On elements of $\Omega^{0,0}(E)$ the connections $\nabla_A$ and $\nabla_{\tilde{A}}$ coincide, so that the above formula is exactly the formula obtained by K\"ahler identities. However, we can make now use of Dirac operator techniques to obtain a sharp estimate.


\section{Proof of the estimate} \label{p}

Since the connection $\nabla_A$ satisfies the Hermitian-Einstein condition, we have that $i \Lambda F_A = c \mb{I}_E$, where $\mb{I}_E$ is the identity endomorphism of $E$ and $$ c = \frac{2 \pi \deg(E)}{(n-1)! \rk(E) \vol(M)} ~.$$
Substituting this into equation (\ref{prop_ide_gen}) we have
\begin{equation}
\bar{\partial}_A^* \bar{\partial}_A = \frac{1}{2} \nabla_{\tilde{A}}^* \nabla_{\tilde{A}} - \frac{\pi \deg(E)}{(n-1)! \rk(E) \vol(M)}~~.
\end{equation}

Now if $\psi_0 \in \Omega^{0}(E)$ is a section such that $\bar{\partial}_A^* \bar{\partial}_A \psi_0 = \lambda \psi_0$ then the above equation gives
\begin{equation}\label{e1}
  \lambda \psi_0 = \frac{1}{2} \nabla_{\tilde{A}}^* \nabla_{\tilde{A}} \psi_0 - \frac{\pi \deg(E)}{(n-1)! \rk(E) \vol(M)} \psi_0 ~.
\end{equation}
Taking the $L_2$-inner product of this equation with $\psi_0$ we are lead to
\begin{equation}
  \lambda \mid \mid \psi_0 \mid \mid_{L_2}^2 = \frac{1}{2} \mid \mid \nabla_{\tilde{A}} \psi_0 \mid \mid_{L_2}^2
  - \frac{\pi \deg(E)}{(n-1)! \rk(E) \vol(M)} \mid \mid \psi_0 \mid \mid_{L_2}^2~.
\end{equation}

Now we must estimate the term $\mid \mid \nabla_{\tilde{A}} \psi_0 \mid \mid_{L_2}^2$. It is a classical fact (see \cite{Ba}) that if we write the twistor operator as
\begin{equation}
  \mc{T}_A = \sum_{k=1}^{m} e_k \otimes \left( \nabla_{\tilde{A},k} + \frac{1}{2n} e_k \cdot D_A \right) ~,
\end{equation}
then we have the following relation with Dirac operator:
\begin{equation}
  \mc{T}_A^* \mc{T}_A = \nabla_{\tilde{A}}^* \nabla_{\tilde{A}} - \frac{1}{2n} D_A^2 ~.
\end{equation}
This immediately implies that
\begin{equation}\label{ln}
  \mid \mid \nabla_{\tilde{A}} \psi_0 \mid \mid_{L_2}^2 \geq \frac{1}{2n} \langle D_A^2 \psi_0 \mid \psi_0 \rangle =
    \frac{\lambda}{n} \mid \mid \psi_0 \mid \mid_{L_2}^2
\end{equation} 
since $D_A^2|_{\Omega^0(E)} = 2 \bar{\partial}_A^* \bar{\partial}_A$. Substituting (\ref{ln}) into (\ref{e1}), we finally obtain
\begin{equation}
  \lambda \geq - \frac{2n}{2n-1}  \frac{\pi \deg(E)}{(n-1)! \rk(E) \vol(M)}
\end{equation}
proving the first part of the Main Theorem.

Notice that if
$$-\frac{2n}{2n-1}\frac{\pi \deg(E)}{(n-1)! \rk(E) \vol(M)}$$
is an eigenvalue of the Dolbeault Laplacian on sections, then the corresponding eigensection $\psi$ satisfies the twistor equation
\begin{equation} \label{t-eq}
  \mc{T}_A \psi = \sum_{k=1}^{m} e_k \otimes \left( \nabla_{A,k} \psi + \frac{1}{2n} e_k \cdot D_A \psi \right) = 0
\end{equation}
The geometric meaning of this equation is not clear yet. However, we know that solutions for this equation do exist at least in certain particular cases, see Section \ref{s} below.


\section{Dirac operators on Riemann surfaces}

In order to establish the assertion about complex Dirac operators on Riemann surfaces, we need the following Lemma. Identifying $\Sc \otimes E \simeq \Omega^{0,*}(E)$, consider the projection operator
$p_0 : \Omega^{0,*}(E) \rightarrow \Omega^{0}(E)$.

\begin{lem}\label{dl}
If $\psi$ be an eigenstate of $D_A$, with non-null eigenvalue $\mu$, on a Riemann surface $\Sigma$, then we have
\begin{equation}
  p_0 \psi = \psi_0 \neq 0
\end{equation}
Furthermore, if $\mu$ is a nonzero eigenvalue of $D_A$, then $\frac{1}{2} \mu^2$ is an eigenvalue of $\bar{\partial}_A^* \bar{\partial}_A$.
\end{lem}
\begin{proof}
On a Riemann surface the twisted spinor bundle is
\begin{equation}
  \Sc \otimes E = \Omega^{0,0}(E) \oplus \Omega^{0,1}(E)
\end{equation}
If $\psi$ is a eigenspinor of $D_A$ then $\psi$ cannot have defined parity. Using the above identification the only way to this happens is if $\psi_0 \neq 0$ and $\psi_1 \neq 0$, in particular we have that $\psi_0 = p_0 \psi \neq 0$.

Now suppose that $\psi \in \Gamma(\mb{S} \otimes E)$ is an spinor such that $D_A \psi = \mu\psi$, where $\mu\neq 0$. Since $\psi_0 \neq 0$, we are lead to conclude that $D_A^2 \psi_0 = \mu^2 \psi_0$, as desired.
\end{proof}

Therefore, comparing with formula (\ref{lb}), we conclude that the nonzero eigenvalues $\nu$ of the complex twisted Dirac operator must satisfy
$$ \mu \geq \sqrt{- \frac{4 \pi \deg(E)}{\rk(E) \vol(M)}}~. $$
completing the proof of the first Corollary.

Finally, recall that in the case of a complex manifold with $\Sp$ structure we can relate the real $\mb{S}$ and complex $\Sc$ spinor bundles by the formula $\Sc = \mb{S} \otimes K_M^{1/2}$. Furthermore, the complex Dirac operator $D$ coincides with a twisted real Dirac $\cal D_S$, where $S$ is the connection on $K_M^{1/2}$ induced by the Chern connection on $M$.

It follows that if we apply our estimate for the nonzero eigenvalues of the complex twisted Dirac operator to the bundle $K_{M}^{1/2}\otimes E$, then we obtain a lower bound for the nonzero eigenvalues of the real twisted Dirac operator on the bundle $E$. Since $\deg(K_{M}^{1/2}\otimes E)=\deg(E)-\rk(E)(1-g)$, we obtain:
$$ \nu \geq \sqrt{\frac{4\pi(1-g)}{\vol(M)}- \frac{4 \pi \deg(E)}{\rk(E) \vol(M)}} \geq \sqrt{\frac{R_0}{2}- \frac{4 \pi \deg(E)}{\rk(E) \vol(M)}}~,$$
where the second inequality follows from the Gauss-Bonnet formula. This completes the proof of the second Corollary.


\section{Sharpness of the estimates}\label{s}

In \cite{AP} the authors computed the spectrum of the real Dirac operator twisted by a connection with constant curvature on a line bundle over a Riemann surface. In particular, they prove \cite[Theorem 5.3]{AP}:

\begin{teo}
Let $L\to\C\PP^1$ be a Hermitian line bundle with a unitary Hermitian connection $\nabla_L$ whose curvature $F_L = -iB\nu$, where $\nu$ is the Riemannian volume form associated to a metric of constant scalar curvature $R$ on $\C\PP^1$. Let $\cal D_L$ is the real Dirac operator twisted by the connection $\nabla_L$. If $\deg(L) \leq 0$, then the spectrum of the operator ${\cal D}_L$, is the set:
\begin{equation}\label{spec-ap}
  \sqrt{\frac{R}{2}\left( (q+1)^2 - (q+1)\deg(L) \right)} ~,~ q\in\Z^+ ~.
\end{equation}\end{teo}

Setting $q=0$ into formula (\ref{spec-ap}), it is easy to see that the estimate of the second Corollary for the nonzero eigenvalues of the twisted real Dirac operator is indeed attained in this example.

Furthermore, consider $L=K_{\C\PP^1}^{-1/2}\otimes E$, so that $\deg(L) = \deg(E) + 1$; let $\nabla_L$ be the tensor connection $\nabla_S \otimes \mb{I} + \mb{I} \otimes \nabla_A$. Then the twisted complex Dirac operator $D_A$ coincides with the twisted real Dirac operator ${\cal D}_L$, and, by formula (\ref{spec-ap}), the spectrum of $D_A$ is given by:
$$ \sqrt{\frac{R}{2}\left( (q+1)^2 - (q+1)(1+\deg(E)) \right)} ~,~ q\in\Z^+ ~.$$
Setting $q=0$, we conclude that the smallest eigenvalue of $D_A$ is \linebreak $\sqrt{-R\deg(E)/2}$, hence the smallest eigenvalue of the twisted Dolbeault Laplacian acting on sections of $E$ is precisely $-R\deg(E)/4$.

Now let us apply the Main Theorem and the Corollary to the case $M=\C\PP^1$ with a metric of constant scalar curvature $R$, so that, by the Gauss-Bonnet formula, $R=8\pi/\vol(M)$. Let $E\to M$ be a line bundle of negative degree, i.e. $\deg(E) \leq -1$. Therefore our estimate (\ref{lb}) for the eigenvalues of the twisted Dolbeault Laplacian can be written as
$$ \lambda \ge - \frac{2\pi \deg(E)}{\vol(M)} = - \frac{R \deg(E)}{4} ~, $$
while estimate (\ref{lb-d}) for the eigenvalues of the twisted Dirac operator can be written as
$$ \lambda \ge \sqrt{- \frac{4\pi \deg(E)}{\vol(M)}} = \sqrt{- \frac{R \deg(E)}{2}} ~, $$

Comparing with the results mentioned above, we conclude that the lower bounds (\ref{lb}) and (\ref{lb-d}) are actually attained in this example. 

Similar considerations and comparison with the results of \cite{AP} allow us to conclude that the lower bounds (\ref{lb}) and (\ref{lb-d}) are also attained when $M$ is a Riemann surface of arbitrary genus with a metric of constant curvature and $E\to M$ is a line bundle of the appropriate degree equipped with a connection with constant curvature (see \cite[Theorems 5.10 and 5.22]{AP}). Notice that the fact that the lower bound (\ref{lb}) is attained also allow us to conclude that the twistor equation (\ref{t-eq}) admits solutions in these cases.


\end{document}